\documentclass[12pt]{article}
\usepackage{amsmath,amsfonts,amssymb,amsthm}
\usepackage{txfonts}%
\usepackage{cases}
\usepackage{bm}
\usepackage{comment}

\theoremstyle{plain}
\newtheorem{theorem}{Theorem}[section]
\newtheorem{lemma}[theorem]{Lemma}

\newtheorem{proposition}[theorem]{Proposition}
\theoremstyle{definition}

\theoremstyle{remark}
\newtheorem{remark}[theorem]{Remark}

\renewcommand{\indent}{\hspace*{5mm}}
\newcommand{\cK}{{\cal K}}

\newcommand{\br}{\gamma}  %
\newcommand{\uc}{\psi^{U}}  %
\newcommand{\lc}{\psi^{L}}  %
\newcommand{\cps}{{\cal K}} %
\newcommand{\mcps}{{\cal L}} %
\newcommand{\bss}{{\cal Q}} %
\newcommand{\bssd}{{\cal M}} %
\newcommand{\dynp}{{\cal N}} %
\newcommand{\nstark}{{n_k'}}

\begin{document}

\title{A game-theoretic proof of Erd\H{o}s-Feller-Kolmogorov-Petrowsky law of the iterated
logarithm for fair-coin tossing}

\author{
Takeyuki Sasai\thanks{Graduate School of Information Science and Technology, University of Tokyo}, \ 
Kenshi Miyabe\thanks{School of Science and Technology, Meiji University} \ 
and Akimichi Takemura\footnotemark[1]\ %
}
\date{November, 2014}
\maketitle

\begin{abstract}
We give a game-theoretic proof of  the celebrated 
Erd\H{o}s-Feller-Kolmogorov-Petrowsky law of the iterated logarithm for fair-coin tossing.  
Our proof, based on Bayesian strategy, is explicit as
many other game-theoretic proofs of the laws in probability theory.
\end{abstract}


\section{Introduction}
\label{sec:intro}
Let $x_n=\pm 1$, \ $n=1,2,\dots$, \ be independent symmetric Bernoulli random variables
with $P(x_n=-1)=P(x_n=1)=1/2$.  Let 
$S_n = x_1 + x_2 + \dots + x_n$.
Concerning the behavior of $S_n$, 
the celebrated 
Erd\H{o}s-Feller-Kolmogorov-Petrowsky law of the iterated logarithm (EFKP-LIL \cite[Chapter 5.2]{Rev90}) states the following:
\begin{equation}
\label{eq:efkp-lil}
P(S_n \ge \sqrt{n}\psi(n) \ \ i.o.)= 0 \ \text{or}\ 1 \quad \text{according as} \quad
\int_1^\infty \psi(\lambda) e^{-\psi(\lambda)^2/2}  \frac{d\lambda}{\lambda}\ <  \infty  \ \text{or} \ =\infty , 
\end{equation}
where $\psi$ is a positive non-decreasing continuous function defined on $[1,\infty)$.
The set of functions $\psi$ such that $P(S_n \ge \sqrt{n}\psi(n) \ \ i.o.)= 0$
is called the {\em upper class} and 
the set of functions $\psi$ such that $P(S_n \ge \sqrt{n}\psi(n) \ \ i.o.)= 1$
is called the {\em lower class} \cite[pp.33-34]{Rev90}.

As the name indicates, this is an extension of the LIL.
The first one who showed this result seems to be Kolmogorov,
which has been stated in L\'evy's book \cite{Lev37} without a proof.
Erd\H{o}s \cite{Erd42} has given a complete proof,
which has been generalized by Feller \cite{Fel43,Fel46} (see also Bai \cite{Bai89}).
Petrowsky \cite{Pet35} has proved the statement for Brownian motion
(see also It\^o and McKean \cite[Section 1.8 and 4.12]{ItoMcK74}
and Knight \cite[Section 5.4]{Kni81}).
Further developments can be seen in the literature
such as similar statements for self-normalized sums \cite{GriKue91,CsoSzyWan03},
for weighted sums \cite{BerHorWeb10}
and for Brownian motion \cite{Kep97}.

In order to state a game-theoretic version of EFKP-LIL, 
consider the following fair-coin game with the initial capital $\alpha>0$.
\begin{quote}
{\sc Fair-Coin Game}\\
\textbf{Players}: Skeptic, Reality\\
\textbf{Protocol}:\\
\indent $\cps_0:=\alpha$.\\
\indent FOR $n=1,2,\ldots$:\\
\indent\indent Skeptic announces $M_n\in\mathbb{R}$.\\
\indent\indent Reality announces $x_n\in\{-1,1\}$.\\
\indent\indent $\cps_n:=\cps_{n-1}+M_n x_n$.\\
\textbf{Collateral Duties}:
Skeptic must keep $\cps_n$ non-negative.
Reality must keep $\cps_n$ from tending to infinity.
\end{quote}
Usually $\alpha$ is taken to be 1, but in Section \ref{sec:sharpness}  we use $\alpha\neq 1$
for notational simplicity.
Let
\begin{equation}
\label{eq:Ipsi}
I(\psi)=\int_1^\infty \psi(\lambda) e^{-\psi(\lambda)^2/2} \frac{d\lambda}{\lambda}.
\end{equation}
The goal of this paper is to prove
the game-theoretic statement of EFKP-LIL in the following form.
\begin{theorem}\label{th:efkp-lil}
Let $\psi$ be a positive non-decreasing continuous function defined on $[1,\infty)$.
In the fair-coin game,
\begin{align}
I(\psi) < \infty  \ 
\Rightarrow & \ \ \text{Skeptic can force}\ S_n < \sqrt{n}\psi(n) \ \  a.a. 
\label{eq:game-validity-integral-condition}
\\
I(\psi) = \infty \ 
\Rightarrow & \ \ \text{Skeptic can force}\ S_n \ge \sqrt{n}\psi(n) \ \  i.o.
\label{eq:sharpness1}
\end{align}
\end{theorem}
The first statement is the {\em validity} and the second statement is the {\em sharpness} of EFKP-LIL.  
For terminology and notions of game-theoretic probability see \cite{shavov01}.
As shown in Chapter 8 of \cite{shavov01}, game-theoretic statement  of EFKP-LIL
in \eqref{eq:game-validity-integral-condition} and 
\eqref{eq:sharpness1} implies the measure-theoretic statement in \eqref{eq:efkp-lil}.
% Furthermore our proof gives a clear relation
% between the almost-sure events and the integrability.  
We use the same line of arguments of the proof of  LIL in Chapter
5 of Shafer and Vovk \cite{shavov01}, but we use some 
Bayesian strategies as ingredients. Our proof shows that the proof by 
Shafer and Vovk can be adapted to prove some stronger forms of LIL.

In Section  \ref{sec:validity} we give a proof of the validity and
in Section \ref{sec:sharpness} we give a proof of the sharpness.  
We discuss some topics for further research in Section \ref{sec:discussion}.

We use the following notation throughout the paper
\[
\ln_k n = \underbrace{\ln \ln \dots \ln}_{k \text{times}} n.
\]

\section{Validity}
\label{sec:validity}
As is often seen in the upper-lower class theory
(see Feller \cite[Lemma 1]{Fel46}),
we can restrict our attention to $\psi$ such that
\begin{align}\label{eq:uc0}
\lc(n)\le\psi(n)\le\uc(n)\mbox{ for all sufficiently large }n,
\end{align}
where
\[
\lc(n)=\sqrt{2 \ln_2 n + 3 \ln_3 n}, \quad
\uc(n)=\sqrt{2 \ln_2 n + 4 \ln_3 n}.
\]
Here $L$ means the lower class and $U$ means the upper class.
It can be verified that $I(\uc)<\infty$ and $I(\lc)=\infty$.
% Note that if a function $\psi(n)$ belongs to the upper class, then any function
% larger than $\psi(n)$ belongs to the upper class,
% and a similar statement holds for the lower class.
%
We discretize the integral in \eqref{eq:Ipsi}
as
\begin{equation}
\label{eq:validity-suml-condition}
\sum_{k=1}^\infty \frac{ \psi(k)}{k} e^{-\psi(k)^2/2} < \infty.
\end{equation}
Since $xe^{-x^2/2}$ is decreasing for $x\ge1$,
the function $\lambda\mapsto\frac{\psi(\lambda)}{\lambda}e^{-\psi(\lambda)^2/2}$ is decreasing for $\lambda$ such that $\psi(\lambda)\ge 1$ 
and 
convergences of the integral in \eqref{eq:Ipsi}
and the sum in \eqref{eq:validity-suml-condition} are equivalent.

\subsection{Constant-proportion betting strategy}

Our proof highly depends on constant-proportion betting strategy (and its mixture).
Here we give basic properties.

We fix a small positive  $\delta$
for the rest of this paper, e.g., $\delta =0.01$.

A constant-proportion betting strategy %$S^\br$ 
with the parameter $\br$ sets
\[
M_n = \br \cps_{n-1}
\]
for a constant $\br \in (-1,1)$.
For the rest of this paper we assume $0\le \br \le \delta$.
The capital process with this strategy is denoted by $\cps^\br_n$.
Note that $\cps^\br_n$ is always positive.
With the initial capital of $\cps^\br_0=\alpha$,
the value $\cps^\br_n$ can be evaluated as
\begin{equation}
\label{eq:cps-2}
\cps^\br_n = \alpha \prod_{i=1}^n (1+\br x_i)= \alpha \left(\frac{1+\br}{1-\br}\right)^{S_n/2} (1-\br^2)^{n/2} .
\end{equation}
Note that $\cps^\br_n$ is determined (except $n$ and $\br$) by $S_n$ and 
is monotone increasing in $S_n$.
In particular, by \eqref{eq:cps-2}, we have
\begin{equation}
\label{eq:negative-capital}
S_n \le 0 \ \ \Rightarrow \ \  \cps^\br_n \le \alpha.
\end{equation}

By the fact that
\[
t - \frac{1}{2}t^2- |t|^3
\le 
\ln(1+t) \le t -  \frac{1}{2}t^2 + |t|^3
\]
for $|t|\le \delta$,
taking the logarithm of $\prod_{i=1}^n (1+\br x_i)$  we have
\begin{equation*}
\br S_n - \frac{1}{2}\br^2 n - \br^3 n \le \ln (\cps_n^\br/\alpha)
\le \br S_n - \frac{1}{2}\br^2 n + \br^3 n
\end{equation*}
and
\begin{equation}
\label{eq:cp-bound2}
e^{-\br^3 n} e^{\br S_n - \br^2 n/2} \le \cps^\br_n/\alpha  \le
e^{ \br^3 n} e^{\br S_n - \br^2 n/2}.
\end{equation}
For the proof of validity, we only use the lower bound in
\eqref{eq:cp-bound2}.

\subsection{Proof of validity}

In this section we let $\alpha=1$.
For notational simplicity we write $\psi_k = \psi(k)$.
The convergence of the infinite series in \eqref{eq:validity-suml-condition} implies the existence of a non-decreasing sequence of positive reals $a_k$ diverging to infinity ($a_k\uparrow \infty$), such that
the series multiplied term by term by $a_k$ is still convergent:
\begin{align}
\label{eq:validity-suml-condition1}
Z:=\sum_{k=1}^\infty a_k\frac{ \psi_k}{k} e^{-\psi_k^2/2}  < \infty.
\end{align}
This is easily seen by dividing the infinite series into blocks of sums less than
or equal to $1/2^k$ and multiplying the $k$-th block by $k$
(see also \cite[Lemma 4.15]{miyabe_convrandseries}).

For $k\ge 1$ let 
\[
p_k = \frac{1}{Z}a_k \frac{\psi_k}{k} e^{-\psi_k^2/2} 
\]
and consider the capital process of a countable mixture of constant-proportion strategies
\begin{align}\label{eq:validity-br}
\cK_n = \sum_{k=1}^\infty p_k \cps_n^{\br_k},
\quad\mbox{ where }\quad
\br_k = \frac{\psi_k}{\sqrt{k}}. 
\end{align}
Obviously $\cK_n$ is never negative.
By the upper bound in \eqref{eq:uc0}, as $k\rightarrow\infty$ we have
\[
\br_k \le \sqrt{\frac{2 \ln_2 k + 4 \ln_3k}{k}} \rightarrow 0. 
\]
Hence $\br_k < \delta$ for sufficiently large $k$.

We now confirm that $\limsup_n \cK_n=\infty$ if 
$S_n \ge \sqrt{n}\psi_n$ infinitely often.
By \eqref{eq:cp-bound2} and \eqref{eq:validity-suml-condition1},
we have
\allowdisplaybreaks
\begin{align*}
Z \cK_n &\ge Z\sum_{k: \br_k < \delta} p_k \exp( \br_k S_n  - \frac{\br_k^2 n}{2} -  \br_k^3 n)\\
&= \sum_{k: \br_k < \delta}  a_k \frac{\psi_k}{k} \exp(-\frac{\psi_k^2}{2} +  \br_k S_n 
-\frac{\br_k^2 n}{2} - \br_k^3 n).
\end{align*}
We consider $n$ and $k$ such that
$S_n \ge \sqrt{n}\psi_n$, \ $\br_k<\delta$, \ $\left\lfloor n-n/\psi_n\right\rfloor\le k\le n$
and $\psi_n/(\psi_n -1)\le 1+\delta/2$.
By \eqref{eq:validity-br},
we have
\begin{align*}
-\frac{\psi_k^2}{2}+ \br_k S_n - \frac{\br_k^2 n}{2}
&\ge-\frac{\psi_k^2}{2}+\sqrt{n}\psi_n\frac{\psi_k}{\sqrt{k}}-\frac{\psi_k^2}{k}\frac{n}{2}\\
&=\psi_k\left(-\frac{1}{2}\left(1+\frac{n}{k}\right)\psi_k+\sqrt{\frac{n}{k}}\psi_n\right)\\
&\ge-\frac{\psi_n^2}{2}\left(\sqrt{\frac{n}{k}}-1\right)^2
\ge-\frac{\psi_n^2}{2}\left(\frac{n}{k}-1\right)^2 \\
&\ge-\frac{1}{2}\left(\frac{\psi_n}{\psi_n-1}\right)^2
\ge-\frac{1}{2}-\delta.
\end{align*}
For sufficiently large $n$, we have
\begin{align*}
\psi_n\le\uc(n)<\uc(2k)=\sqrt{2\ln_2( 2k)+4\ln_3(2k)}
 <2\sqrt{2\ln_2 k+3\ln_3 k}=2\lc(k)\le2\psi_k.
\end{align*}
Thus,
\begin{align*}
Z \cK_n 
&\ge  \sum_{k=\left\lfloor n-n/\psi_n \right\rfloor}^n a_k \frac{\psi_k}{k}
\exp(-\frac{1}{2}-\delta - \br_k^3 n)\\
&\ge  a_{\left\lfloor n-n/\psi_n \right\rfloor} \frac{\psi_n}{2n} \sum_{k=\left\lfloor n-n/\psi_n \right\rfloor}^n 
\exp(-\frac{1}{2}-\delta -\br_n^3 n)\\
&\ge  a_{\left\lfloor n-n/\psi_n \right\rfloor} \frac{\psi_n}{2n} \left(\frac{n}{\psi_n}-1\right)\exp(-\frac{1}{2}-\delta -\br_n^3 n)\\
&=  a_{\left\lfloor n-n/\psi_n \right\rfloor} \left(\frac{1}{2}-\frac{\psi_n}{2n}\right)\exp(-\frac{1}{2}-\delta -\br_n^3 n).
\end{align*}
Since $a_{\left\lfloor n-n/\psi_n \right\rfloor}\rightarrow\infty$, \ 
$\psi_n/n\to0$  and $\br_n^3 n\rightarrow 0$, we have shown
\[
S_n \ge \sqrt{n}\psi_n  \ i.o.\ \Rightarrow \ \limsup_{n\rightarrow\infty} \cK_n = \infty.
\]

\section{Sharpness}
\label{sec:sharpness}

In this section we prove the sharpness
\eqref{eq:sharpness1} of EFKP-LIL in game-theoretic probability, following the approach in Chapter 5 of \cite{shavov01} and \cite{miyabe-takemura-lil}.
We divide our proof into several subsections.
For notational simplicity we use the initial capital of $\alpha=1-2/e=(e-2)/e$ in this section.

\subsection{Change of time scale}
\label{subsec:regularity-conditions}

The first key of our proof is a change of time scale 
from $\lambda$ to $k$:
\[
\lambda= e^{5k \ln k}= k^{5k}.
\]

\begin{remark} 
\label{rem:1}
``5'' in $5k\ln k$ is not essential and we use this time scale for simplicity.
\end{remark}

By taking the derivative of  $\ln \lambda= 5 k \ln k$, we have
\[
\frac{d\lambda}{\lambda} =  5(\ln k+1)dk.
\]
Hence the integrability condition is written as
\begin{equation*}
\int_1^\infty \psi(\lambda) e^{-\psi(\lambda)^2/2} \frac{d\lambda}{\lambda}  = \infty
\  \Leftrightarrow \ 
\int_1^\infty (\ln k) \psi(e^{5k\ln k}) e^{-\psi(e^{5k\ln k})^2/2} dk = \infty.
\end{equation*}
Let $f(x)=\psi(e^{5x\ln x}) e^{-\psi(e^{5x\ln x})^2/2}$.
Since $xe^{-x^2/2}$ is decreasing for $x\ge 1$,
the function $f(x)$ is decreasing for $x$ such that $\psi(e^{5x\ln x})\ge 1$.
Thus, for sufficiently large $k$ and $x$ such that $k\le x\le k+1$, we have
\begin{align*}
\ln x f(x)\ge& \ln k f(x+1)\ge \frac{1}{2}\ln (k+1) f(k+1),\\
\ln x f(x)\le& \ln (k+1) f(x)\le 2\ln k f(k).
\end{align*}
Hence, we have
\begin{equation*}
\int_1^\infty (\ln k) \psi(e^{5k\ln k}) e^{-\psi(e^{5k\ln k})^2/2} dk = \infty 
\ \ \Leftrightarrow \ \ 
\sum_{k=1}^\infty  (\ln k) \psi(e^{5k\ln k}) e^{-\psi(e^{5k\ln k})^2/2} = \infty .
\end{equation*}
Then, it suffices to show that
\begin{equation*}
\sum_{k=1}^\infty  (\ln k) \psi(e^{5k\ln k}) e^{-\psi(e^{5k\ln k})^2/2} = \infty 
\ \  \Rightarrow  \ \ 
\text{Skeptic can force}\ S_n > \sqrt{n} \psi(n)\   i.o.
\end{equation*}
Recall that we can assume \eqref{eq:uc0} here again.

\subsection{Bounding relevant capital processes}
\label{subsec:mixture-strategy}
In this section we introduce mixtures of constant-proportion betting strategies
and bound their capital processes.
We discuss relevant capital processes in further subsections.

\subsubsection{Uniform mixture of constant-proportion betting strategies}
\label{subsubsec:uniform-mixture}
We consider a continuous uniform mixture of constant-proportion strategies with 
the betting proportion $u\br$, $2/e\le u\le 1$. This is a Bayesian strategy, a similar one to which has been considered in \cite{kumon-bayesian}. 

Define
\[
\mcps^\br_n 
=\int_{2/e}^1 \prod_{i=1}^n (1+u\br x_i)du, \qquad \mcps^\br_0=\alpha=1-e/2.
\]
At round $n$ this strategy bets
$
M_n = \int_{2/e}^1 u\br   \prod_{i=1}^{n-1} (1+u\br x_i) du .
$
Then by induction on $n$, the capital process is indeed written as
\begin{align*}
\mcps^\br_n 
&= 
\mcps^\br_{n-1} + M_n x_n =\int_{2/e}^1 \prod_{i=1}^{n-1} (1+u\br x_i) du
+ x_n\int_{2/e}^1 u\br\prod_{i=1}^{n-1} (1+u\br x_i) du \\ 
&=   
\int_{2/e}^1 \prod_{i=1}^n (1+u\br x_i)du.  %
\end{align*}
Applying \eqref{eq:cp-bound2} and noting $u^3 \br^3 \le \br^3$,  we have
\begin{equation*}
 e^{ -\br^3 n} \int_{2/e}^1 e^{u\br S_n - u^2\br^2 n/2}du 
 \le
\mcps^\br_n 
\le  
e^{ \br^3 n}\int_{2/e}^1 e^{u\br S_n - u^2\br^2 n/2}  du .
\end{equation*}

We further bound the integral in the following lemma.

\begin{lemma}\label{lem:mcps-upper} %
\begin{numcases}
{\mcps^\br_n
\le }
\min\left\{1,e^{\br^3 n}\frac{\sqrt{2\pi}}{\br\sqrt{n}}\right\}
& \text{if}\quad $S_n\le 0$, \label{eq:mcps-negative-sn}
\\
 e^{\br^3 n} e^{2\br(S_n/e - \br n/e^2)}
& \text{if}\quad $0<S_n\le 2\br n/e$,
\label{eq:mcps-positive1-sn}\\
 e^{\br^3 n}\min \left\{ e^{S_n^2/(2n)} \frac{\sqrt{2\pi}}{\br \sqrt{n}}, e^{\br S_n/2} \right\}
&  \text{if}\quad $2\br n/e<S_n< \br n$,
\label{eq:mcps-positive2-sn}
\\
e^{\br^3 n}\min\left\{e^{S_n^2/(2n)}\frac{\sqrt{2\pi}}{\br\sqrt{n}},e^{\br S_n-\br^2 n/2}\right\}& \text{if}\quad $S_n \ge  \br n$.
\label{eq:mcps-positive3-sn}
\end{numcases}

\end{lemma}

\begin{proof}
Completing the square we have
\begin{equation}
\label{eq:completing-the-square}
- \frac{1}{2}u^2\br^2 n  + u\br S_n = -\frac{\br^2 n}{2}  \left(u-\frac{S_n}{\br n}\right)^2 + 
\frac{S_n^2}{2n}.
\end{equation}
Hence  by the change of variables
\[
v = \br \sqrt{n} \left( u-\frac{S_n}{\br n}\right), \qquad  du = \frac{dv}{\br \sqrt{n}},
\]
we obtain
\begin{align*}
%\label{eq:mcps-trance}
\int_{2/e}^1 e^{u \br S_n - u^2\br^2 n/2}  du 
&=
e^{S_n^2/(2n)}
\int_{2/e}^1 \exp \left(-\frac{\br^2 n}{2}  \left(u-\frac{S_n}{\br n}\right)^2\right) du\nonumber \\
&= 
e^{S_n^2/(2n)}\frac{1}{\br \sqrt{n}} \int_{2 \br \sqrt{n}/e-S_n/\sqrt{n}}^{\br \sqrt{n} - S_n/\sqrt{n}}
e^{-v^2/2} dv .
\end{align*}
Then for all cases we can bound $\mcps^\br_n$ from above as
\begin{equation}
\label{eq:Q-bound-up-1}
\mcps^\br_n \le  e^{\br^3 n + S_n^2/(2n)} \frac{\sqrt{2\pi}}{\br \sqrt{n}}.
\end{equation}

Without change of variables, we can also bound 
the integral $\int_{2/e}^1 g(u)du$, $g(u)=e^{u \br S_n - u^2\br^2 n/2}$, directly as
\[
\int_{2/e}^1 g(u)du \le
\max_{2/e\le u\le 1} g(u).
\]
Note that 
\begin{equation}
g(2/e)=e ^{2 \br (S_n/e - \br n /e^2)}, \quad g(1)=e^{ \br S_n-\br^2n/2}.
\label{eq:g-lower-upper}
\end{equation}
We now consider the four cases.
\begin{description}
\item[Case 1] $S_n \le 0$. \ 
In this case 
\[
\int_{2/e}^1 e^{u \br S_n - u^2\br^2 n/2}  du 
\le
 \int_{2/e}^1 e^{- u^2\br^2 n/2} du \le 
\frac{\sqrt{2\pi}}{\br\sqrt{n}}.
\]
Also by \eqref{eq:negative-capital}, $\mcps^\br_n  \le 1-2/e\le 1$  if $S_n \le 0$.
This gives \eqref{eq:mcps-negative-sn}.
\item[Case 2] $0 < S_n \le 2 \br n/e$. \ In this case $S_n/(\br n) \le 2/e$ and
by the unimodality of $g(u)$ we have $\max_{2/e\le u \le 1}g(u)= g(2/e)$. Hence 
\eqref{eq:mcps-positive1-sn} follows from \eqref{eq:g-lower-upper}.
\item[Case 3] $2 \br n/e < S_n < \br n$. \  %
In this case  $\max_{2/e\le u \le 1} g(u)=g(S_n/(\br n))=e^{S_n^2/(2n)}$ and
$\mcps^\br_n 
\le
e^{\br^3 n} e^{S_n^2/(2n)}$.
Furthermore in this case $S_n^2 < \br n S_n$ implies $S_n^2/(2n) <  \br S_n/2$
and we also have
\begin{equation}
\label{eq:qn1-case2a}
\mcps^\br_n  
\le
e^{\br^3 n} e^{\br S_n/2}.
\end{equation}
By \eqref{eq:Q-bound-up-1} %
and \eqref{eq:qn1-case2a}, 
we have \eqref{eq:mcps-positive2-sn}.
\item[Case 4] $S_n \ge \br n$.  \ Then $S_n/(\br n)\ge 1$ %
and $\max_{2/e\le u \le 1} g(u)=g(1)$. Hence
\begin{equation}
\label{eq:qn1-case1}
\mcps^\br_n  
\le
e^{\br^3 n} e^{ \br S_n-\br^2n/2}. %
\end{equation}
By \eqref{eq:Q-bound-up-1} and \eqref{eq:qn1-case1},
we have \eqref{eq:mcps-positive3-sn}.
\end{description} 
\end{proof}

\subsubsection{Buying a process and selling a process}
\label{subsubsec:bss}
Next we consider the following capital process.
\begin{equation*}
\bss_n^{\br} =  2 \mcps_n^\br  - \cps_n^{\br e}.
\end{equation*}
This capital process consists of buying two units of $\mcps_n^\br$ and
selling one unit of $\cps_n^{\br e}$. 
This combination of selling and buying is essential in the game-theoretic proof of the law of the iterated logarithm
in Chapter 5 of \cite{shavov01} and \cite{miyabe-takemura-lil}.
However, unlike Chapter 5 of \cite{shavov01} and \cite{miyabe-takemura-lil}, where
combination of {\em three} processes is used, we only combine {\em two} capital processes.

We want to bound $\bss_n^{\br}$ from above. 

\begin{lemma}\label{lem:bss-bound}
Let 
\begin{align}
\label{eq:case-i-negative-condition-1}
C_1 
&= 2 e^{\br^3 n} %\max \left\{1, \exp \left(\frac{(2e-1)(\br^3(1+e^3)n + \ln 2}{(e-1)^2}\right) \right\}.
\exp \left(\frac{(2e-1)(\br^3(1+e^3)n + \ln 2)}{(e-1)^2}\right).
\end{align}
Then
\begin{numcases}
{\bss_n^\br 
\le }
2 \min\left\{1, e^{\br^3 n} \frac{\sqrt{2\pi}}{\br \sqrt{n}}\right\} 
& \text{if}\quad $S_n\le 0$, \label{eq:bss-negative-sn}
\\
 C_1 
& \text{if}\quad $0<S_n\le \br n/e$, \label{eq:bss-positive1-sn}\\
2 e^{\br^3 n}\min \left\{ e^{S_n^2/(2n)} \frac{\sqrt{2\pi}}{\br \sqrt{n}}, e^{\br S_n} \right\}
&  \text{if}\quad $\br n/e<S_n< e\br n$,
\label{eq:bss-positive2-sn}
\\
C_1 
& \text{if}\quad $S_n \ge e \br n$.
\label{eq:bss-positive3-sn}
\end{numcases}
\end{lemma}

\begin{remark}
\label{rem:2}
In this lemma, $C_1$ depends on $\br$ and $n$ through $\br^3 n$.
However from Section \ref{subsec:dynamic-strategy} on, we take
$\br^3 n$ to be sufficiently small. Hence $C_1$ can be taken to be
a constant (cf.\ \eqref{eq:C1-fixed}) not depending on $\br$ and $n$. 
Also note that the interval for $S_n$ in
\eqref{eq:bss-positive2-sn} is larger than the interval in \eqref{eq:mcps-positive2-sn}.
\end{remark}

\begin{proof}
We bound $\bss_n^\br =2\mcps_n^\br - \cps_n^{\br e}$ from above
in the following four cases:
\[
\text{(i)} \ S_n \le 0, \quad
\text{(ii)}\ 0< S_n \le \br n /e,  \quad 
\text{(iii)}\ \br n/e < S_n < e \br n, \quad 
\text{(iv)}\ S_n \ge e  \br n,
\]
\begin{description}
\item[Case (i)]  By $\bss_n^\br \le 2 \mcps_n^\br$, 
\eqref{eq:bss-negative-sn} follows  from \eqref{eq:mcps-negative-sn}.
\item[Case (ii)]\ In this case $S_n/e - \br n /e^2\le 0$. Hence
\eqref{eq:bss-positive1-sn} follows from \eqref{eq:mcps-positive1-sn} and $\bss_n^\br \le 2 \mcps_n^\br$.

\item[Case (iii)] 
We again use $\bss^\br_n\le 2\mcps^\br_n$.
If $\br n/e < S_n \le 2 \br n/e$, then 
\[
\frac{S_n}{e} - \frac{\br n }{e^2}\le \frac{\br n}{e^2} \le  \frac{S_n}{e}
\]
and 
$\mcps_n^\br\le e^{\br^3 n}e^{2\br S_n/e} \le e^{\br^3 n}e^{\br S_n}$
from \eqref{eq:mcps-positive1-sn}.
Otherwise \eqref{eq:bss-positive2-sn} follows from 
\eqref{eq:mcps-positive2-sn} and \eqref{eq:mcps-positive3-sn}.
\begin{comment}
For this  we again use \eqref{eq:Q-bound-up-1} for 
$\mcps_n^\br$ in $\bss_n^\br \le 3  \mcps_n^\br$. By \eqref{eq:Q-bound-up-1} and \eqref{eq:mcps-positive-sn}
\begin{equation*}
\bss_n^\br \le 3 e^{\br^3 n} \min \left\{ e^{S_n^2/(2n)} \frac{\sqrt{2\pi}}{\br \sqrt{n}}, e^{\br S_n} \right\}.
\end{equation*}
\end{comment}

\item[Case (iv)] \ 
Since $S_n \ge e n \br > n \br$, by \eqref{eq:qn1-case1}
we have $\mcps^\br_n  \le e^{\br^3 n} e^{ \br S_n-\br^2n/2}$
and 
\begin{align*}
\bss_n^\br
 &\le 
 2\mcps_n^\br - \cps_n^{\br e}
\le 2  e^{\br^3 n} e^{ \br S_n-\br^2n/2} - e^{-\br^3 e^3n}e^{\br e S_n - \br^2 e^2 n/2}\\
&= 
2e^{\br ^3 n} e^{ \br S_n-\br^2n/2} 
\left( 1 - \frac{1}{2} e^{-\br^3(1+e^3)n}e^{\br (e-1)S_n - (e^2-1)\br^2n/2}\right). 
\end{align*}
Hence if the right-hand side is non-positive we have $\bss_n^\br \le 0$:
\begin{align}
&S_n \ge e n \br \ \  \text{and}\ -\br^3(1+e^3)n - \ln 2 + \br(e-1)S_n - \frac{1}{2}(e^2-1)\br^2n \ge 0 \nonumber\\
& \qquad \qquad \Rightarrow \ \ \bss_n^\br \le 0.
\label{eq:case-iii-1}
\end{align}
Otherwise, write  $A=\br^3(1+e^3)n + \ln 2$ and 
consider the case 
\[
\br(e-1)S_n - \frac{1}{2}(e^2-1)\br^2n \le  A.
\] 
Dividing this by $e-1$ 
and also considering $S_n \ge e n\br$, we have
\begin{align}
\label{eq:case-iii-1a}
\br S_n - \frac{1}{2}(e+1)\br^2n &\le  \frac{A}{e-1},\\
-S_n +en \br &\le 0. 
\label{eq:case-iii-1b}
\end{align}
$\br \times \eqref{eq:case-iii-1b} + \eqref{eq:case-iii-1a}$ gives
\[
\frac{1}{2}(e-1) \br^2 n \le \frac{A}{e-1} \quad \text{or} \quad \frac{1}{2} \br^2 n \le \frac{A}{(e-1)^2}.
\]
Then by \eqref{eq:case-iii-1a}
\[
\br S_n - \frac{1}{2}\br^2 n \le \frac{A}{e-1} + \frac{e}{2} \br^2 n \le
\frac{A}{e-1} + \frac{eA}{(e-1)^2}=\frac{(2e-1)A}{(e-1)^2}.
\]
Hence just using $\bss_n^\br\le 2\mcps_n^\br$ and \eqref{eq:qn1-case1} in this case, we obtain
\begin{equation}
\label{eq:case-iii-conclusion}
\bss_n^\br  \le 2 e^{\br^3 n} \exp\left(\frac{(2e-1)(\br^3(1+e^3)n + \ln 2 )}{(e-1)^2}\right) = C_1.
\end{equation}
This also covers \eqref{eq:case-iii-1}
% because $(2e-1)(\br^3(1+e^3)n + \ln 2  )/(e-1)^2$ is positive. Hence 
and we have \eqref{eq:case-iii-conclusion} for the whole case (iv).
\end{description}
\end{proof}

\subsubsection{Further continuous mixture of processes}
\label{subsubsec:further-mixture}
We finally introduce another continuous mixture of capital processes.
Define a capital process
\begin{equation}
\label{eq:bssd}
\bssd_n^{\br,k}= \frac{1}{\ln k}\int_0^{\ln k} \bss_n^{\br e^{-w}} dw =  \frac{1}{\ln k}\int_0^{\ln k} (2\mcps_n^{\br e^{-w}} - \cps_n^{\br e^{-w+1}})dw.
\end{equation}
% For example $(1/\ln k)\int_0^{\ln k} \mcps_n^{\br e^{-w}} dw$ is the capital process for the strategy
% betting
% \[
% M_n = \frac{1}{\ln k} \int_0^{\ln k} \int_{2/e}^1 ue^{-w}\br \prod_{i=1}^{n-1} (1+ue^{-w} \br x_i)\, du\,  dw
% \]
% at round $n$.
Under the same notation as in Lemma \ref{lem:bss-bound},
if $S_n >0$, we have
\begin{equation}
\label{eq:bssd-bound}
\bssd_n^{\br,k} 
\le 
 C_1 + \frac{2}{\ln k} 
\max_{\br'\in [\br k^{-1},\br]} \left(2 e^{{\br'}^3 n} \min \{ e^{S_n^2/(2n)} \frac{\sqrt{2\pi}}{\br' \sqrt{n}}, e^{\br' S_n} \}\right),
\end{equation}
because the length of the interval
\[
\left \{ w \mid \frac{S_n}{ne} < \br e^{-w} < \frac{S_n e}{n}\right \}
\]
is equal to 2.

If $S_n \le 0$, since $\max_{0\le w\le \ln k} e^{\br^3 e^{-3w}n}/e^{-w} \le e^{\br^3 n}k$,
integrating \eqref{eq:bss-negative-sn}, we have 
\begin{equation}
\bssd_n^{\br,k} 
\le 
\frac{2}{\ln k} \int_0^{\ln k}  \min\left\{1, e^{\br^3 e^{-3w}n} \frac{\sqrt{2\pi}}{\br e^{-w} \sqrt{n}}\right\} dw
% \nonumber \\
% \le 2
% \max_{w\in[0,\ln k]} \min  \left\{1, e^{\br^3 e^{-3w}n} \frac{\sqrt{2\pi}}{\br e^{-w} \sqrt{n}}\right\} 
% \nonumber \\
%&
\le  2
\min  \left\{1, \frac{\sqrt{2\pi}}{\br \sqrt{n}}e^{\br^3 n} k \right\}.
\label{eq:bssd-bound-neg}
\end{equation}
% Hence
% \begin{equation}
% S_n \le 0 \ \Rightarrow \ \bssd_n^{\br,k} \le 
% 2 \min  \left\{, \frac{\sqrt{2\pi}}{\br \sqrt{n}}e^{\br^3 n} k  \right\}.
% \end{equation}

Also note that $\int_0^{\ln k} \mcps_n^{\br e^{-w}} dw$ is bounded by the  double  integral
\begin{equation}
\label{eq:double-integral}
\int_0^{\ln k} \mcps_n^{\br e^{-w}} dw
%= \frac{1}{\ln k} \int_0^{\ln k}\int_{2/e}^1  \prod_{i=1}^{n-1} (1+ue^{-w} \br x_i)\, du\,  dw
\le e^{\br^3 n} \int_0^{\ln k}
\int_{2/e}^1  e^{ue^{-w} \br S_n - u^2 e^{-2w} \br^2n /2}  du\,  dw.
\end{equation}

\subsection{Dynamic strategy forcing the sharpness}
\label{subsec:dynamic-strategy}
Write
\[
n_k = %
e^{5k\ln k} = k^{5k}, \ \ \psi_k = \psi(n_k).
\]
Note that $\psi_k$ here is different from $\psi_k$ in Section \ref{sec:validity}.
As in Chapter 5 of \cite{shavov01} and \cite{miyabe-takemura-lil},
we divide the time axis into  ``cycles'' $[n_k, n_{k+1}]$, $k\ge 1$.
Betting strategy for the $k$-th cycle is based on the following betting proportion:
\begin{equation}
\label{eq:br_k}
\br_k= \frac{\psi_{k+1}}{\sqrt{n_{k+1}} }k^2.
\end{equation}

% \begin{remark}
% \label{rem:2}
% ``2'' in $k^2$ for $\br_k$ is chosen to satisfy less than half of  ``5'' in $e^{5k\ln k}$.
% \end{remark}

As a preliminary consideration we check %the relation between $\br_k$ and $e$
the growth of $n_k$, $\psi_k$ and $\br_k$.

\begin{lemma}
\label{lem:psi-growth}
\begin{align}
\lim_{k\rightarrow\infty} \br_k^3 n_{k+1} &= 0,
\label{eq:br3-psi}
\\
\lim_{k\rightarrow\infty} \frac{\uc(n_k)}{\psi_{k+1}}&= 1,
\label{eq:uck1}
\\
\lim_{k\rightarrow\infty} \frac{k^5 n_k}{n_{k+1}} &= e^{-5}, 
\label{eq:nknk1}
\\
\lim_{k\rightarrow\infty}  \br_k \sqrt{n_k}\psi_{k+1} &=0.
\label{eq:brk-sqrtnk}
\end{align}
\end{lemma}

\begin{proof}
%Consider $\br_k^3 (1+e^3) n_{k+1}$.  
By $\psi(n)\le \uc(n)$ for sufficiently large $n$, we have
%such that $\ln k/\ln (k+1) \ge 1-\delta
\[
\br_k^3 n_{k+1} 
= 
\frac{\psi_{k+1}^3}{\sqrt{n_{k+1}}}k^6 
\le  %
\frac{(2\ln_2 n_{k+1} + 4 \ln_3 n_{k+1})^{3/2}}
{(k+1)^{5(k+1)/2}} k^6 \rightarrow 0 \qquad (k\rightarrow\infty)
\]
and \eqref{eq:br3-psi} holds.
All of $\uc(n_k)$, $\uc(n_{k+1})$, $\lc(n_k)$, $\lc(n_{k+1})$
are of the order 
\begin{equation}
\label{eq:order-of-nk}
\sqrt{2 \ln \ln e^{5k\ln k}}(1+o(1))=\sqrt{2\ln k}(1+o(1))
\end{equation}
as $k\rightarrow\infty$ and \eqref{eq:uck1} holds by \eqref{eq:uc0}.
\eqref{eq:nknk1} holds because
\[
\lim_{k\rightarrow\infty} \frac{k^{5(k+1)}}{(k+1)^{5(k+1)}} =\lim_{k\rightarrow\infty} \left(1-\frac{1}{k+1}\right)^{5(k+1)}
= e^{-5}.
\]
Then $n_k/n_{k+1}=O(k^{-5})$ and \eqref{eq:brk-sqrtnk} holds by

\[
\br_k \sqrt{n_k}\psi_{k+1} = \psi^2_{k+1}k^2 (n_k/n_{k+1})^{1/2} \rightarrow 0 \qquad(k\rightarrow\infty).
\]
\end{proof}

By \eqref{eq:br3-psi},
we choose $k_0$ such that  $\br_k^3 (1+e^3) n_{k+1} \le \delta$ for $k\ge k_0$.
Then in our formulas in the previous section, in the $k$-th cycle,
we have 
\[
e^{\br_k^3 n}\le  e^{\br_k^3 n_{k+1}} \le e^\delta.
\]
Furthermore $C_1$ in \eqref{eq:case-i-negative-condition-1} can be taken  as
\begin{align}
C_1 %&\le 2 e^{\br^3 n} \max \left\{ 1, \exp \left(\frac{(2e-1)(\br^3(1+e^3)n + \ln 2)}{(e-1)^2}\right) \right\} \\
&=2e^{\delta}\exp \left(\frac{(2e-1)(\delta (1+e^3) + \ln 2}{(e-1)^2}\right).
\label{eq:C1-fixed}
\end{align}

%Now we check the growth of $\psi_k$ and $\br_k$.
We also take $k_0$ sufficiently large such that left-hand sides of 
\eqref{eq:uck1},
\eqref{eq:nknk1} and
\eqref{eq:brk-sqrtnk} are within $\pm \delta$ of their respective limits.
% $\uc(n_k)/\psi_{k+1}\le 1+\delta$
% for all $k\ge k_0$.

For each cycle $[n_k, n_{k+1}]$, $k\ge k_0$, 
we apply the following capital process 
to $x_n$'s in the cycle.
\begin{equation}
\label{eq:dynp}
\dynp_{n}^{\br_k} =\alpha + \frac{1}{D} (\ln k) \psi_{k+1} e^{-\psi_{k+1}^2/2} (\alpha-\bssd_{n-n_k}^{\br_k,k}), 
\qquad \alpha=1-2/e, \ D=\frac{32\sqrt{2\pi}}{\alpha}.
\end{equation}
Here we gave a specific value of $D$ for definiteness, but from the proof below it will be clear
that any sufficiently large $D$ can be used.
Since the strategy for $\bssd_{n-n_k}^{\br_k,k}$ is applied only to $x_n$'s in the cycle, 
$\alpha=\dynp_{n_k}^{\br_k}=\bssd_{0}^{\br_k}$.
Concerning $\dynp_{n}^{\br_k}$ we prove the following proposition.
% Recall that we have a fixed $\delta<0.01$ and chose $k_0$ such that 
% $\br_k^3 (1+e^3)n_{k+1} \le \delta$ and $\uc(n_k)/\psi_{k+1}\le 1+\delta$ for $k\ge k_0$.
% Now, we set $D=8\sqrt{2\pi}/(1-\delta)$ in \eqref{eq:dynp}
In its proof we increase $k_0$ to satisfy 
all of \eqref{eq:C_1to0}--\eqref{eq:stop}, if needed.

\begin{proposition}
\label{prop:dynp}
Suppose that  
\begin{equation}
\label{eq:within-range}
-\sqrt{n} \uc(n) \le S_n \le \sqrt{n} \psi(n), \qquad  \forall n\in [n_k,n_{k+1}]. 
\end{equation}
for all sufficiently large $k$. 
Then there exists $k_0$ such that for all $k\ge k_0$
\begin{equation}
\label{eq:nkn-lower-bound}
\dynp_n^{\br_k} > \frac{\alpha}{2}, %\delta^2, 
\ \forall n\in [n_k,n_{k+1}],
\end{equation}
and
\begin{equation}
\label{eq:nk+1} 
\dynp_{n_{k+1}}^{\br_k} \ge \alpha \left(1 + \frac{1-\delta}{D} (\ln k) \psi_{k+1} e^{-\psi_{k+1}^2/2}\right).
\end{equation}
\end{proposition}

\begin{proof}
In our proof, for notational simplicity we write $n$ instead of $n-n_k$.

For proving \eqref{eq:nkn-lower-bound}, 
we use \eqref{eq:bssd-bound} and \eqref{eq:bssd-bound-neg} for  $S_n - S_{n_k}$. 
We want to bound $\bssd_{n}^{\br_k,k}$ from above.
Note that 
$(\ln k)\psi_{k+1} e^{-\psi_{k+1}^2/2}\rightarrow 0$ as $k\rightarrow\infty$
by \eqref{eq:order-of-nk}.
Hence the case $S_n\le 0$ is trivial because $\bssd_{n}^{\br_k,k}$ is bounded from above 
by \eqref{eq:bssd-bound-neg}.

For the case $S_n > 0$, by the term $\dfrac{2}{\ln k}$ on the right-hand side of \eqref{eq:bssd-bound}, 
it suffices to show % the following: 
\begin{align*}
&0< S_n \le \sqrt{n_k} \uc(n_k) + \sqrt{n_k + n} \psi(n_k+ n)  \\
& \quad \Rightarrow \  
\psi_{k+1} e^{-\psi_{k+1}^2/2} 
% \max_{\br\in [\br_k k^{-1},\br_k]} \left(2 e^{\delta} \min \{ e^{S_n^2/(2n)} \frac{\sqrt{2\pi}}{\br \sqrt{n}}, e^{\br S_n} \}\right)
\times 2 e^{\delta} \min \{ e^{S_n^2/(2n)} \frac{\sqrt{2\pi}}{\br \sqrt{n}}, e^{\br S_n} \}
< \frac{D\alpha}{4} -\delta, \ \  
\forall \br_k \in[k^{-1},\br_k], \ 
\forall n\in [0,n_{k+1}-n_k],
\end{align*}
for sufficiently large $k$ such that 
\begin{equation}
\label{eq:C_1to0}
C_1 (\ln k) \psi_{k+1} e^{-\psi_{k+1}^2/2} < \delta.
\end{equation}
% Otherwise, we prove 
% \begin{align*}
% &0 < S_n \le \sqrt{n_k} \uc(n_k) + \sqrt{n} \psi(n_k+ n)  \\
% & \quad \Rightarrow \   
% \psi_{k+1} e^{-\psi_{k+1}^2/2} \times
% 2 e^{{\br}^3 n}\min \{ e^{S_n^2/(2n)} \frac{\sqrt{2\pi}}{\br \sqrt{n}}, e^{\br S_n} \}\le \frac{D}{4}(1-\delta^2),\\
% & \qquad \qquad  \qquad 
%  \forall n\in [0,n_{k+1}-n_k], \forall \br\in [\br_k k^{-1},\br_k].
% \end{align*}
Let
\begin{equation}
\label{eq:c1-const}
c_1 =\frac{9}{(1+2\delta)^2} \quad \text{s.t.} \quad 
\frac{1}{2} - \frac{1}{\sqrt{c_1}} - \delta > 0.
\end{equation}
We distinguish two cases: 
\[
\text{(a)} \ n\le \frac{\psi_{k+1}^2}{c_1 \br^2}, \quad \text{(b)}\  \frac{\psi_{k+1}^2}{c_1 \br^2} < n \le n_{k+1}-n_k.
\] 

For case (a),  $\sqrt{n_k} \uc(n_k) \le (1+\delta)\sqrt{n_k}\psi_{k+1}$ by \eqref{eq:uck1}
for sufficiently large $k$.
Also $\psi(n_k+n)\le \psi(n_{k+1})$. 
Hence in this case 
% \[
% S_n \le \left((1+\delta) \sqrt{n_k} + \sqrt{n_k +\psi_{k+1}^2/(c_1\br^2) } \right)\psi_{k+1}
% \]
% and
\[
\br S_n  \le  \left((1+\delta)\br \sqrt{n_k} + \sqrt{\br^2n_k + \psi_{k+1}^2/c_1}\right)\psi_{k+1}.
\]
Then for $\br \le \br_k$, by \eqref{eq:brk-sqrtnk}
\begin{equation}
\label{eq:another-large-k}
\br S_n \le  \left((1+\delta)\br_k\sqrt{n_k} + \sqrt{\br_k^2 n_k+ \psi_{k+1}^2/c_1}\right)\psi_{k+1}
\le  \frac{\psi_{k+1}^2}{\sqrt{c_1}}(1+\delta) 
\end{equation}
for sufficiently large $k$.
Since %that for sufficiently large $k$
\begin{equation}
\label{eq:case_a_con} 
\psi_{k+1} e^{-\psi_{k+1}^2/2} \times 2 e^{\delta} e^{\br S_n} 
\le 2 \psi_{k+1}  e^{\delta} \exp\left(-\psi_{k+1}^2\big(\frac{1}{2} - \frac{1}{\sqrt{c_1}}-\delta\big)\right) \rightarrow 0 \qquad (k\rightarrow\infty), 
%< \frac{D}{4} (1-\delta^2) 
\end{equation}
we have
$\dynp_n^{\br_k} > \alpha/2$ uniformly in $\br\in [\br_k k^{-1},\br_k]$ and in $n\le \psi_{k+1}^2/c_1\br^2$.

For case (b), $\psi_{k+1}/\sqrt{c_1}\le \br \sqrt{n}$ and $S_n \le ((1+\delta)\sqrt{n_k} + \sqrt{n_k + n})\psi_{k+1}$.
Hence 
\begin{align*}
\psi_{k+1} e^{-\psi_{k+1}^2/2} \times 
2 e^{\delta}e^{S_n^2/(2n)} \frac{\sqrt{2\pi}}{\br \sqrt{n}}  
& \le
\psi_{k+1} e^{-\psi_{k+1}^2/2} \times \frac{2e^\delta\sqrt{2\pi}\sqrt{c_1}}{\psi_{k+1}}
\exp\left(\frac{\big((1+\delta)\sqrt{n_k} + \sqrt{n_k + n}\big)^2}{2n}\psi_{k+1}^2\right) \\
&= 
2 e^\delta\sqrt{2\pi}\sqrt{c_1} \exp\left( \frac{(1+(1+\delta)^2)n_k + 2(1+\delta)\sqrt{n_k}\sqrt{n_k+n}}{2n} \psi_{k+1}^2\right).
%\\
% &= 
% 2 e^\delta \sqrt{2\pi} \sqrt{c_1} \exp\left( \Big(\frac{(1+(1+\delta)^2)n_k}{2n} + (1+\delta)\sqrt{\frac{n_k}{n}}
% \sqrt{1 + \frac{n_k}{n}} \Big)\psi_{k+1}^2\right).
\end{align*}
For $\br\le \br_k$
\[
\frac{\psi_{k+1}^2}{c_1 \br^2} <  n \ \Rightarrow \ \frac{n_k}{n} \psi_{k+1}^2
 < c_1 \br^2 n_k  \le c_1 \br_k^2 n_k = c_1 \frac{n_k}{n_{k+1}} k^4 \psi_{k+1}^2.
\]
Hence $\psi_{k+1}^2 n_k/n \rightarrow 0$ as $k\rightarrow \infty$. 
Similarly $\sqrt{n_k/k} \psi_{k+1}^2 \rightarrow 0$ as $k\rightarrow \infty$. 
% Also  by \eqref{eq:brk-sqrtnk},
% as $k\rightarrow\infty$
% \[
% \sqrt{\frac{n_k}{n}} \psi_{k+1}^2 \le \sqrt{c_1} \br_k \sqrt{n_k} \psi_{k+1} \le \frac{3\sqrt{c_1} \ln_2 e^{5(k+1)\ln (k+1)}}{\sqrt{k}} \rightarrow 0.
% \]
Therefore, for sufficiently large $k$,
\begin{equation}
\psi_{k+1} e^{-\psi_{k+1}^2/2} \times 
2 e^{\delta}e^{S_n^2/(2n)} \frac{\sqrt{2\pi}}{\br \sqrt{n}} 
% & \le 
% 2 e^\delta \sqrt{2\pi}\exp\left( \frac{1+(1+\delta)^2}{2} \Big(\frac{n_k}{n} + \sqrt{\frac{n_k}{n}}
% \sqrt{1 + \frac{n_k}{n}}\Big) \psi_{k+1}^2\right)  \nonumber\\
\le 2 e^{2\delta} \sqrt{2\pi}\sqrt{c_1}
<  \frac{D\alpha}{4}-\delta,
%\frac{D}{4}(1-\delta^2).
\label{eq:case_b_con}
\end{equation}
with the choice of $D$ in  \eqref{eq:dynp} and $c_1$ in \eqref{eq:c1-const}.
% The left-hand side is bounded uniformly in 
% $\br \le \br_k$ and $n\in [n_k/\br^2,n_{k+1}-n_k]$.
This proves \eqref{eq:nkn-lower-bound}.

Now we prove \eqref{eq:nk+1}. Write $\nstark =n_{k+1}-n_k$. 
We will show that $\bssd_{\nstark }^{\br_k,k}\rightarrow 0$ as $k\rightarrow\infty$ if
% By \eqref{eq:mcps-trance}, $\br_k e^{-w}\sqrt{\nstark }/e \le 2 \br_k e^{-w} \sqrt{\nstark } - S_{\nstark }/\sqrt{\nstark }  \, \left(\, \forall w \in [0, \ln k] \, \right)$, $k^2n_k/n_{k+1} \rightarrow 0 \,\, \left(k \rightarrow \infty\right)$ and 
\[
S_{\nstark } \le 2\sqrt{n_k} \uc(n_k) + \sqrt{\nstark } \psi(\nstark ) \le 2 \sqrt{\nstark } \psi(\nstark ).
\] 
Let $y=u e^{-w}$. Then 
$2/(ek) \le y \le 1$ in the integral \eqref{eq:double-integral}.
As in \eqref{eq:completing-the-square}, % with $u$ replaced by $y$, 
the quadratic function  $y \br_k S_{\nstark } - y^2 \br_k^2 \nstark /2$ 
is maximized at $y=S_{\nstark }/(\br_k \nstark )$.  Now
\begin{equation}
\frac{S_{\nstark }}{\br_k \nstark } \le \frac{2 \psi(\nstark )}{\br_k \sqrt{\nstark }} \le
\frac{2 \psi(\nstark )}{\psi_{k+1}} \sqrt{\frac{n_{k+1}}{\nstark }} \frac{1}{k^2} \le 2(1+\delta) \frac{1}{k^2} < \frac{2}{ek}
\label{eq:y-lower-bound}
\end{equation}
for sufficiently large $k$, because $\lim_{k\rightarrow\infty} n_{k+1}/\nstark  = 1$.
Then the integrand in \eqref{eq:double-integral} is maximized at $y=ue^{-w}=2/(ek)$ and we have
\begin{align}
%\label{eq:stop}
\bssd_{\nstark }^{\br_k,k}  
%&\le  \frac{1}{\ln k} \int^{\ln k}_0 2 \mcps_{\nstark }^{\br e^{-w}}dw \\
% &\le 
% 2 e^\delta\frac{1}{\ln k}\int_0^{\ln k} 
% \int_{2/e}^1 \exp(ue^{-w}\br_kS_{\nstark } -u^2e^{-2w} \br_k^2  \nstark /2)du dw \nonumber \\
% &\le
% 2e^{\br_k^3 \nstark }\frac{1}{\ln k}\int_0^{\ln k} \frac{e^{S_{\nstark }^2/(2\nstark )}}{\br_k {e^{-w}} \sqrt{\nstark }}  \int_{2 \br_k e^{-w} \sqrt{\nstark } /e - S_{\nstark } /\sqrt{\nstark }}^{\br_k e^{-w} \sqrt{\nstark } - S_{\nstark }/\sqrt{\nstark } } e^{-v^2/2}dvdw \nonumber \\
% &\le
%  2e^\delta\frac{1}{\ln k}\int_0^{\ln k} \frac{e^{S_{\nstark }^2/(2\nstark )}}{\br_k {e^{-w}} \sqrt{\nstark }} \cdot \frac{e-2}{e} \cdot \br_k e^{-w} \sqrt{\nstark } e^{-\br_k^2 e^{-2w} \nstark /(2e^2)} dw \nonumber \\
&\le 
2 e^\delta \alpha \exp\left( \frac{2}{ek} \br_k  S_{\nstark } - \left(\frac{2}{ek}\right)^2 \br_k^2  \nstark /2\right)
\le 
2 e^\delta\alpha \exp\left( \frac{4}{ek} \br_k  \sqrt{\nstark }\psi(\nstark ) 
- \frac{2}{e^2k^2} \br_k^2 \nstark \right) 
\nonumber \\
% &\le
% 2 \alpha e^\delta e^{(2/(ek)) \br_k 2 \sqrt{\nstark } \psi(\nstark ) - \br_k^2  \nstark /(2e^2k^2)} \nonumber \\
% &\le
% 2\frac{e-2}{e} e^{\br_k^3 \nstark } e^{\psi_{k+1}^2 \left(4-k^2/e^2+ k^2n_k/(e^2 n_{k+1}) \right)/2} \le  \frac{e-2}{e}-D\delta\quad \left(k\rightarrow \infty \right)\\
&\le 2 e^\delta\alpha \exp\left(\frac{4}{e} k\psi_{k+1}^2 \sqrt{\frac{\nstark }{n_{k+1}}}
-\frac{2}{e^2} k^2 \psi_{k+1}^2 \frac{\nstark }{n_{k+1}}\right) \rightarrow 0  \quad (k\rightarrow 0).
\label{eq:stop}
\end{align}
%
% \begin{align}
% \label{eq:stop}
% \bssd_{\nstark }^{\br_k}  \le  \frac{1}{\ln k} \int^{\ln k}_0 2 \mcps_{\nstark }^{\br e^{-w}}dw 
% &\le 
% 2 e^{\br_k^3 \nstark }\frac{1}{\ln k}\int_0^{\ln k}  \int_{2/e}^1 e^{u\br_ke^{-w}S_{\nstark } -u^2 \br_k^2 e^{-2w} \nstark /2}du dw \nonumber \\
% &\le
% 2e^{\br_k^3 \nstark }\frac{1}{\ln k}\int_0^{\ln k} \frac{e^{S_{\nstark }^2/(2\nstark )}}{\br_k {e^{-w}} \sqrt{\nstark }}  \int_{2 \br_k e^{-w} \sqrt{\nstark } /e - S_{\nstark } /\sqrt{\nstark }}^{\br_k e^{-w} \sqrt{\nstark } - S_{\nstark }/\sqrt{\nstark } } e^{-v^2/2}dvdw \nonumber \\
% &\le
%  2e^{\br_k^3 \nstark }\frac{1}{\ln k}\int_0^{\ln k} \frac{e^{S_{\nstark }^2/(2\nstark )}}{\br_k {e^{-w}} \sqrt{\nstark }} \cdot \frac{e-2}{e} \cdot \br_k e^{-w} \sqrt{\nstark } e^{-\br_k^2 e^{-2w} \nstark /(2e^2)} dw \nonumber \\
% &\le
% 2\frac{e-2}{e} e^{\br_k^3 \nstark }e^{S_{\nstark }^2/(2\nstark )-\br_k^2  \nstark /(2e^2k^2)} \nonumber \\
% &\le
% 2\frac{e-2}{e} e^{\br_k^3 \nstark } e^{\psi_{k+1}^2 \left(4-k^2/e^2+ k^2n_k/(e^2 n_{k+1}) \right)/2} \le  \frac{e-2}{e}-D\delta\quad \left(k\rightarrow \infty \right).
% \end{align}
%
\end{proof}

We assume that by the validity result, Skeptic already employs a strategy forcing
$S_n \ge -\sqrt{n}\uc(n)\ a.a.$  
In addition to this strategy, based on Proposition \ref{prop:dynp}, 
consider the following strategy.
\begin{quote}
Start with initial capital $\cps_0=\alpha$.\\
Set $k=k_0$.\\
Do the followings repeatedly:\\
\indent 1.  Apply the strategy in Proposition \ref{prop:dynp} for $n\in [n_k, n_{k+1}]$. \\ 
\indent \quad If \eqref{eq:within-range} holds, then go to 2. Otherwise go to 3.\\
\indent 2. Let $k=k+1$. Go to 1.\\
\indent 3. Wait until $\exists k'$ such that $-\sqrt{n_{k'}} \uc(n_{k'}) \le S_{n_{k'}} \le
\sqrt{n_{k'}}\psi(n_{k'})$.\\ 
\indent \quad Set $k=k'$ and go to 1.
\end{quote}

Since Skeptic already employs a strategy forcing 
$S_n \ge -\sqrt{n} \uc(n)\ a.a.$, the lower bound in \eqref{eq:within-range} violated only
finite number of times.  Hence if $S_n \le \sqrt{n}\psi(n) \ a.a.$, then
Step 3 is performed only finite number of times.  Also when
Step 3 is performed, the overshoot of $|x_n|=1$ does not make Skeptic bankrupt by
\eqref{eq:nkn-lower-bound}. Now for each iteration of Step 2, Skeptic
multiplies his capital at least by 
\[
1 + \frac{1-\delta}{D} (\ln k) \psi_{k+1} e^{-\psi_{k+1}^2/2}.
\]
Then
\[
\frac{1-\delta}{D} \sum_{k=k_0}^\infty (\ln k)\psi_{k+1} e^{-\psi_{k+1}^2/2}
  \le \prod_{k=k_0}^\infty  \left( 1 + \frac{1-\delta}{D} (\ln k)\psi_{k+1} e^{-\psi_{k+1}^2/2}\right).
\]
Since the left-hand side diverges to infinity, the above strategy forces the sharpness.

\section{Discussion}
\label{sec:discussion}
In this paper we gave a game-theoretic proof of EFKP-LIL. Our validity proof is very short.
Our sharpness proof is elementary, but it is still long.  
A simpler sharpness proof is desired.

In our sharpness proof we used the change of time scale 
$\lambda =e^{5k\ln k}$ and formed the cycles $[n_k, n_{k+1}]$
based on this time scale.
There is a question whether this scale is the best.  Actually we can prove
the sharpness based on the change of time scale $\lambda =e^{ck\ln_2 k}$ for large $c$.
Any sparser cycles than $n_k = e^{ck\ln_2 k}$ can be used for proving the sharpness.

It is interesting to consider a generalization of EFKP-LIL to games other than the fair-coin game.
In particular the case of self-normalized sums discussed in \cite{GriKue91,CsoSzyWan03} is 
also important from game-theoretic viewpoint.
Self-normalized sums in game-theoretic probability have been studied in \cite{kumon-etal-sos}.

Another possible extension is that $\psi(n)$ is sequentially given by a third player
Forecaster at the beginning of each round of the game.  From the game-theoretic viewpoint it is
of interest to ask whether Skeptic can force
\[
\sum_{n=1}^\infty \frac{ \psi(n)}{n} e^{-\psi(n)^2/2} = \infty \ \ 
\Leftrightarrow \ \ 
S_n \ge \sqrt{n}\psi(n) \ \ i.o.
\]

\bibliographystyle{abbrv}
\bibliography{efkp}

\end{document}